\documentclass{amsart}
\usepackage[all]{xy}
\usepackage{amsmath,amssymb}
\def\id{\mathrm{id}}

\def\C{\mathrm{Cone}}

\def\q{\mathbb{Z}_q}

\newtheorem{thm}{Theorem}[section]
\newtheorem{cor}[thm]{Corollary}
\newtheorem{lem}[thm]{Lemma}
\newtheorem{prop}[thm]{Proposition}

\theoremstyle{definition}
\newtheorem{defn}[thm]{Definition}
\theoremstyle{remark}

\numberwithin{equation}{section}

\def\t{\mathbb{T}}

\def\q{\mathbb{Q}}

\def\zq{\mathbb{Z}_q}

\begin{document}
\title [Finite and Torsion $KK$-theories]
{Finite and Torsion $KK$-theories}%
\author{Hvedri Inassaridze \;\;\  and \;\;\ Tamaz Kandelaki }%
\address{H. Inassaridze , T.Kandelaki:\\
Tbilisi Centre for Mathematical Sciences, A. Razmadze Mathematical Institute, M.
Alexidze Str. 1, 380093 Tbilisi, Georgia}%
\email{}%

\thanks{The authors were partially supported by GNSF-grant, INTAS-Caucasus grant and Volkswagen Foundation grant}%

\begin{abstract}

 We develop a finite $KK^{G}$-theory of $C*$-algebras following Arlettaz-H.Inassaridze's approach to finite algebraic
$K$-theory \cite{arin}
. The Browder-Karoubi-Lambre's theorem on
the orders of the elements for finite algebraic $K$-theory [ , ]
is extended to finite $KK^{G}$-theory. A new bivariant theory,
called torsion $KK$-theory is defined as the direct limit of
finite $KK$-theories. Such bivariant $K$-theory has almost all
$KK^G$-theory properties  and one has  the following exact
sequence
\begin{multline*}
\cdots \rightarrow KK^{G}_n(A,B)\rightarrow
KK^{G}_n(A,B;\mathbb{Q})\rightarrow
KK^{G}_{n}(A,B;\mathbb{T})\rightarrow \cdots
\end{multline*}
relating $KK$-theory, rational bivariant $K$-theory and torsion $KK$-theory.
 For a given homology theory on the category of separable
$GC^*$ -algebras finite, rational and torsion homology theories are introduced and investigated. In particular, we formulate finite, torsion and rational versions of Baum-Connes Conjecture.
The later is equivalent to the investigation of rational and
$q$-finite analogues for Baum-Connes Conjecture for all prime
$q$.
\end{abstract}

\maketitle
\section*{Introduction}

In this paper we provide a new bivariant theory, which
will be called torsion equivariant $KK^G$-theory. That is closely
connected with the usual and rational versions of $KK^G$-theories.
By definition torsion $KK^G$-theory is a direct limit of
$KK^G$-theory with coefficients in $Z_q$ ($q$-finite $KK^G$-theory
in our terminology), where $q$ runs over all natural numbers $\geq 2$.
This new bivariant homology theory has all the properties of
$KK^G$-theory except of the existence of the identity morphism.
 We arrive to the following principle: some of
problems that arise in usual $KK^G$-theory may be reduced to
suitable problems in rational, finite and torsion $KK^G$-theories.
Namely, it will be shown that Baum-Connes conjecture has analogues
in finite, torsion and rational $KK$-theories and Baum-Connes
assembly map is an isomorphism if and only if its rational and
finite assembly maps are isomorphisms for all prime $q$ (Theorem
\ref{bcceq}).

As a technical tool, we mainly work with homology theories on the
category of $C^*$-algebras with action of a fix locally compact
group $G$. In sections 1 and 2 for a given homology theory $H$
torsion and $q$-finite homology theories $H^{(q)}$ are constructed
and their properties are investigated. Much of these properties
are known for experts in some concrete form, but we could not find
suitable references for our purposes. They are redefined and
reinvestigated here. Furthermore, in section 2 we define and
investigate a new homology theory, so called torsion homology
theory. Especially, we make accent on the following twosided long
exact sequence of abelian groups
$$
\dots \rightarrow H_{n+1}^{\mathbb{T}}(A)\rightarrow
H_n(A)\xrightarrow{r} H_n(A)\otimes \mathbb{Q}\rightarrow
H_n^{\mathbb{T}}(A)\rightarrow H_{n-1}(A)\xrightarrow{r}\cdots
$$ for any $GC^*$-algebra $A$ which is used concretely for bivariant $KK$-theories in the sequel
section. In particular, based on results of these sections we
list properties of torsion and finite bivariant
$KK^G$-theories. Besides, there exists a long exact sequence,
which is similar to the above long exact sequence:
\begin{multline}\label{excmain}
\cdots \rightarrow KK^{G
}_{n+1}(A,B;\mathbb{Q})\rightarrow KK^G_{n+1}(A,B;\mathbb{Q}/\mathbb{Z})\rightarrow \\
\rightarrow KK^G_n(A,B)\xrightarrow{Rat_n}
KK^G_n(A,B;\mathbb{Q})\rightarrow
KK^G_{n}(A,B;\mathbb{Q}/\mathbb{Z}))\rightarrow \cdots
\end{multline}
The similar result for $K$-theory of bornological algebras one can
find in \cite{cumero}.

The rational bivariant $KK$-theory and the torsion bivariant $KK$-theory
have all the properties of usual
bivariant $KK$-theory. The only difference is that the torsion case
hasn't unital morphisms. Note that rational bivariant $KK$-theory
 used in this paper differs from the similar one defined in \cite{bla}.

In the next section 3 we study torsion and $q$-finite
$KK$-theories, where the finite $KK$-theory is redefined following
Arlettaz-H.Inassaridze's approach to finite algebraic $K$-theory
\cite{arin}. Sections 4 and 5 are devoted to the proof of  the
following Browder-Karoubi-Lambre' theorem for finite $KK$-theory
(see Theorem 5.5):

  Let $A$ and $B$ be, respectively, separable and
$\sigma $-unital $C^*$-algebras, real or complex; and $G$ be a
metrizable compact group. Then, for all integer $n$,
\begin{enumerate}
    \item $\;\;\;\;$ $q\cdot KK_n^G (A,B;Z/q) = 0,$ $\;\;\;\;\;\;$ if $q-2$ is not divided by $4$;
    \item $\;\;\;\;$ $2q \cdot KK_n ^G(A,B;Z/q) = 0,$ $\;\;\;\;\;$ if $4$ divides $q-2$.
    \end{enumerate}

\bigskip

It is clear that this result holds for non-unital rings too. For finite algebraic $K$-theory this
theorem for $n = 1$ was proved algebraically by Karoubi and Lambre \cite{}, and for
$n> 1$ by Browder \cite{brw}.

 The key idea to carry out this
problem is its reduction to the algebraic $K$-theory case.
This is realized by two steps. First we calculate finite
topological $K$-theory of $C^*$-algebras and additive
$C^*$-categories by finite algebraic $K$-theory of rings. Then
generalizing the main result of \cite{kan1}, finite bivariant
$KK^G$-theory is calculated by finite topological $K$-theory of the
additive $C^*$-category of Fredholm modules.
When $G$ is a locally compact group, it is more complicated to get the similar result for finite
$G$-equivariant bivariant $KK$-theory and we
intend to investigate this problem in a forthcoming paper.

\section{On Finite homology theory}

\bigskip

In this section we analyze some properties of homology
theory with coefficients in $\mathbb{Z}_q$  which is said to be
$q$-\textit{finite homology}. There exist some different ways to
construct for a given homology theory on $C^*$-algebras a corresponding
$q$-finite homology theory; we choose one of them, suitable for
 our purposes.

\bigskip

Let $S^1$ be the unit cycle in the plane of complex numbers with
module one. The map
$$
\tilde{q}:S^1\rightarrow S^1,\;\;\;\;\;\;\;\;x\mapsto x^q,
$$
$q\geq 2$, $q\in \mathbb{N}$, is called \textit{standard q-th
power map}. Since $1\in S^1$ is invariant relative to the map
$\tilde{q}$, it can be considered as a map of pointed spaces
$$
\tilde{q}:S^1_*\rightarrow S^1_*,\;\;\;\;\;\;\;\;x\mapsto x^q,
$$
where $*=1$. These are basic  q-th power maps in algebra and
topology.

Let $C_0(S^1)$ be a $C^*$-algebra of continuous complex (or real)
functions on the unit cycle $S^1$ in the plane of complex numbers
with module one vanishing at $1$. Then the map
$$
\tilde{q}:S^1\rightarrow S^1,\;\;\;\;\;\;\;\;x\mapsto x^q,
$$
$q\geq 2$, $q\in \mathbb{N}$, induces a $*$-homomorphism
$$
\hat{q}:C_0(S^1)\rightarrow C_0(S^1),\;\;\;\;\;\;\;f(s)\mapsto
f(s^q).
$$
Denote $C^*$-algebra $C_q$ as cone of the homomorphism $\hat{q}$:
$$
C_q =\{(x,f)\in C_0(S^1)\oplus C_0(S^1)\otimes
C[0;1)\;|\;\;\hat{q} (x)=f(0)\},
$$

The following lemma is one of the main property of the degree map.
The idea of the proof is taken from \cite{weid}.

\begin{lem}
Let $p_q:C_{pq}\rightarrow C_q$ be a natural map induced by a
commutative
$$
\xymatrix{C_0(S^1)\ar[r]^p\ar[d]_{pq} & C_0(S^1)\ar[d]^q\\
C_0(S^1)\ar[r]^=&C_0(S^1)}
$$
diagram. Then there is a natural homomorphism $\nu
_{p,q}:C_{p_q}\rightarrow C_p$, which is a homotopy equivalence.
\end{lem}

\begin{proof}
A homomorphism $\nu _{p,q}$ is induced by the commutative diagram
$$
\xymatrix{C_{pq}\ar[d]\ar[r]^{p_q} & C_{q}\ar[d]\\
C_0(S^1)\ar[r]^p&C_0(S^1).}
$$
Choose a homotopy $H:[0,1]^2\times [0,1]\rightarrow [0,1]^2$
relative to $L=[0,1]\times \{1\}\cup \{0,1\}\times [0,1]$ such
that $H_0=id$ and $H_1$ is the retraction of $[0,1]^2$ on $L$.
Then, a homotopy inverse to $\nu _{p,q}$  is given by
$$
\chi :C_p\rightarrow C_{p_q},\;\;\;\;\chi (a,b)=(a, b,
\tilde{b}\cdot H_1)
$$
For $b\in C_0(S^1)[0,1)$ define $\tilde{b}:L\rightarrow C_0(S^1)$
by $\tilde{b}(s,1)=(1,t)=0$ and $\tilde{b}(0,t)=b(t^q),\;\;t,s\in
[0,1]$. We have $\nu _{p,q} \cdot \chi =C_p$.

There is a homotopy between $id_{C_{p_q}}$ and $\chi \cdot \nu
_{p,q}$ which is given by a map
$$G_t(a,b,c)=(a,b,c_t)$$
where $c_t(r,s)=c(H_t(r,s))$.
\end{proof}

\begin{lem}
\label{contenz} \begin{enumerate}
    \item Let $A$ be a $C^*$-algebra.  Then
the commutative diagram
$$
\xymatrix{A\otimes C_q\ar[r]\ar[d] & A\otimes C_0(S^1)^{[0,1)\ar[d]}\\
A\otimes C_0(S^1)\ar[r]^{id_A\otimes \hat{q}}& A\otimes C_0(S^1)}
$$
is a pullback diagram. In particular, $A\otimes C_q\simeq
C_{id_A\otimes \hat{q}} $.
    \item Let the diagram
$$
\xymatrix{A\ar[r]\ar[d] & B\ar[d]\\
C\ar[r]^{}& D}
$$
be a pullback diagram and $(X,x)$ pointed compact space. Then the
induced diagram
$$
\xymatrix{A^{(X,x)}\ar[r]\ar[d] & B^{(X,x)}\ar[d]\\
C^{(X,x)}\ar[r]^{}& D^{(X,x)}}
$$
is a pullback diagram.
\end{enumerate}
\end{lem}

\begin{proof}
(1). Let the diagram
$$
\xymatrix{P \ar[r]\ar[d]& A\otimes C_0(S^1)^{[0,1)\ar[d]}\\
A\otimes C_0(S^1)\ar[r]^{id_A\otimes \hat{q}}& A\otimes C_0(S^1)}
$$
be a pullback diagram. Then $P$ contains the couple of functions
$(f(s),g(s,t))$, such that $f(s^q)=g(s,0)$, $f(0)=0$, $g(0,t)=0$;
$s\in S^1$ $t\in [0,1)$. Therefore the pair  $(f(s),g(s,t))$
defines a continuous function on the cone $\sigma _q$ of the
degree map $S^1\xrightarrow{q} S^1$ with values in $A$ . So, there
is a homomorphism  $P\rightarrow A^{\sigma _q}\simeq A\otimes C_q
$ (which is a morphism of suitable diagrams). Thus the diagram is
pullback and as a consequence we get the isomorphism $A\otimes C_q\simeq
C_{id_A\otimes \hat{q}} $.

(2) is trivial.
\end{proof}
\bigskip

Recall that a family of functors $H=\{H_n\}_{n\in \mathbb{Z}}$
on the category of (separable or $\sigma $-unital )
$GC^*$-algebras (real or complex) \cite{kas1} is said to be
homology theory (cf. \cite{bla}):
 if
 \begin{enumerate}
    \item $H_n$ is  a homotopy invariant functor for any $n\in Z$
    \item for any $*$-homomorphism ($G$-equivariant ) of $\sigma $-unital
    algebras $f:A\rightarrow B$
    there exists a natural twosided
   long exact sequence of abelian groups:
   $$
\dots \rightarrow H_{n+1}(B)\rightarrow H_n(C_f)\rightarrow
H_n(A)\rightarrow H_n(B)\rightarrow H_{n-1}(C_f)\rightarrow \cdots
   $$
   where $C_f$ is the cone of $f$.
\end{enumerate}

\bigskip

\begin{defn}
\label{modq} Under the $q$-finite homology of a homology $H$, $q\geq
2$, we mean a family of functors $H^{(q)}=\{H^{(q)}_n\}_{n\in
\mathbb{Z}}$, where
$$H^{(q)}_n=H_{n-2}(-\otimes C_q).$$
\end{defn}

\bigskip
Below we list  main properties of the $q$-finite homology.

\begin{prop}
\label{propf} Let $H$ be a homology theory on the category of $C^*$-algebras. Then
\begin{enumerate}
    \item $H^{(q)}$ is a homology theory;
    \item there is a twosided long exact sequence of abelian
    groups
    $$
\dots \rightarrow H_{n+1}^{(q)}(A)\rightarrow
H_n(A)\xrightarrow{q\;\cdot} H_n(A)\rightarrow
H_n^{(q)}(A)\rightarrow H_{n-1}(A)\xrightarrow{q\;\cdot}\cdots
    $$
    \item there is a twosided long exact sequence of abelian
    groups
    $$
\dots \rightarrow H_{n+1}^{(p)}(A)\rightarrow
H_n(A)^{(q)}\xrightarrow{\dot{p}}
H_n^{(pq)}(A)\xrightarrow{\acute{q}} H_n^{(p)}(A)\rightarrow
H_{n-1}^{(q)}(A)\xrightarrow{\dot{p}}\cdots
    $$
    \item if homology $H$ has an associative product $$H_n(A)\otimes H_m(B)\rightarrow H_{n+m}(A\otimes
B)$$ then there is an associative product
    $$
H_n^{(p)}(A)\otimes H_{q}^{(q)}(B)\rightarrow
H_{n+m-2}^{(pq)}(A\otimes B).
    $$
\end{enumerate}
\end{prop}

\begin{proof} The first and the second parts are immediate consequences of Lemma \ref{contenz}
and  Definition \ref{modq}. The third is a trivial consequence of
the Puppe's exact sequence for the homomorphism $p_q:A\otimes
C_{pq}\rightarrow A\otimes C_q$, Lemma \ref{propf} and Lemma
\ref{contenz}. Finally we have
\begin{multline}
H_n^{(p)}(A)\otimes H_{m}^{(q)}(B)=H_{n-2}(A\otimes C_p)\otimes
H_{m-2}^{(q)}(B\otimes C_q)\rightarrow\\\rightarrow
H_{n+m-4}(A\otimes B \otimes C_p \otimes C_q)\rightarrow
H_{n+m-4}(A\otimes B \otimes C_{pq}) \rightarrow
H_{n+m-2}^{(pq)}(A\otimes B)
\end{multline}
where the product $C_p\otimes C_q\rightarrow C_{pq}$ is defined as
follows. There are natural homomorphisms $\breve{q}:C_p\rightarrow
C_{pq}$, induced by the commutative diagram
$$
\xymatrix{C_p\ar[r]\ar[d]_{\breve{q}} & C_0(S^1) \ar[r]^p\ar[d]^=
& C_0(S^1)\ar[d]^q\\C_{pq}\ar[r] & C_0(S^1) \ar[r]^{pq} &
C_0(S^1),}
$$
and similarly $\breve{p}:C_q\rightarrow C_{pq}$. Since all
algebras are nuclear (in $C^*$-algebraic sense), these
homomorphisms yield a homomorphism (product) $C_p\otimes
C_q\rightarrow C_{pq}$ which is associative in the obvious sense.
\end{proof}

\section{On Torsion Homology theory}
Now we define a new homology theory using the family of $q$-finite
homology theories, $q\geq 2$. Consider the ordered set
$\mathbb{N}_{(2)}=\{q\in \mathbb{N}\;|\; q\geq 2\}$, where $q\leq
q'$ iff $q$ divides $q'$.

Note that if $q'=qs$, then there is a natural transformation of
functors $$\tau ^{(qq')}_n:H^{(q)}_n\rightarrow H^{(q')}_n$$
induced by the homomorphism $q_s$,
$$
\xymatrix{C_q\ar[r]\ar[d]_{q_s} & C_0(S^1) \ar[r]^q\ar[d]^= &
C_0(S^1)\ar[d]^s\\C_{q'}\ar[r] & C_0(S^1) \ar[r]^{q'} & C_0(S^1),}
$$
where $\tau ^{(qq')}_n(A):H^{(q)}_n(A)\rightarrow H^{(q')}_n(A)$
denotes the homomorphism $H_n(id_A\otimes q_s)$. Therefore one has an
inductive system of abelian groups
$$
\{H^{(q)}_n(A), \tau ^{(qq')}_n(A) \}_{q\in \mathbb{N}_{(2)}}
$$
for any $GC^*$-algebra $A$.

\bigskip

\begin{prop}
Let $H$ be a homology theory and $H^{\mathbb{T}}$ be a family of functors
defined by the equality
$$
H^{\mathbb{T}}_n(A)=\underrightarrow{\lim}_q \;\;H^{(q)}_n(A),
\;\;\; n\in \mathbb{Z},\;\;\; q\geq 2,
$$
for any $C^*$-algebra $A$. Then
\begin{enumerate}
    \item $H^{\mathbb{T}}$ is a homology theory $H$ on the category
    $GC^*$-algebras;
    \item There is a twosided long exact sequence of abelian
    groups
$$
\dots \rightarrow H_{n+1}^{\mathbb{T}}(A)\rightarrow
H_n(A)\xrightarrow{r} H_n(A)\otimes \mathbb{Q}\rightarrow
H_n^{\mathbb{T}}(A)\rightarrow H_{n-1}(A)\xrightarrow{r}\cdots
    $$
for any $GC^*$-algebra $A$.
     \item There is a twosided long exact
sequence of abelian groups
$$
\dots \rightarrow H_{n+1}^{\mathbb{T}}(A)\xrightarrow{\hat{q}}
H_{n+1}^{(q)}(A)\xrightarrow{}
H_n^{\mathbb{T}}(A)\xrightarrow{\grave{p}}
H_n^{\mathbb{T}}(A)\xrightarrow{\hat{q}}
H_{n}^{(q)}(A)\xrightarrow{}\cdots
    $$
for any $GC^*$-algebra $A$.
\end{enumerate}

\end{prop}

\begin{proof}
According to Proposition \ref{propf} (1) the first part is an easy consequence of the fact that
the direct limit preserves homotopy and excision properties. For the
second part consider the commutative diagram
$$
\xymatrix{
  \cdots \ar[r]&H_{n+1}^{(q)}(A)\ar[r]\ar[d]& H_n(A) \ar[r]^{q}\ar @{=}[d] & H_n(A)
  \ar[r]\ar[d]^{\frac{q'}{q}}& H_n^{(q)}(A)\ar[r]\ar[d]^{\tau ^{(qq')}_n(A)} &\cdots  \\
  \cdots \ar[r]&H_{n+1}^{(q')}(A)\ar[r]& H_n(A) \ar[r]^{q'} & H_n(A)
  \ar[r]& H_n^{(q')}(A)\ar[r] &\cdots ,}
$$
where rows are long twosided exact sequences. By taking the direct limit of these long exact sequences,
one gets the following long twosided exact sequence
$$
\xymatrix{\cdots \ar[r]& H_n(A) \ar[r]^{\hat{q}\;\;\;\;} &
\underrightarrow{\lim}_q H_n(A)
  \ar[r]& H_n^{\mathbb{T}}(A)\ar[r] &H_{n-1}(A)\ar[r] &\cdots }
$$
It is easy to check that the inductive system $\{H_n(A),
\frac{q'}{q}\}$ is isomorphic to the inductive system
$\{H_n(A)\otimes \mathbb{Z}^{\{q\}}, \frac{q'}{q}\}$ , where
$\mathbb{Z}^{\{q\}}= \mathbb{Z}$ for all $q$. Then
$$
\underrightarrow{\lim}_q H_n(A)\simeq H_n(A)\otimes
\underrightarrow{\lim}_q \;\mathbb{Z}^{\{q\}}\simeq H_n(A)\otimes
\mathbb{Q},
$$
since one has the isomorphism $\underrightarrow{\lim}_q \;\mathbb{Z}^{\{q\}}\simeq
\mathbb{Q}$ defined by the map $(q,r)\mapsto \frac{r}{q}$.

For (3) consider the commutative diagram
$$
\xymatrix{
  \cdots \ar[r]&H_{n+1}^{(q)}(A)\ar[r]\ar[d]& H_n^{(p)}(A) \ar[r]^{p}\ar[d] & H_n^{(pq)}(A)
  \ar[r]\ar[d]^{\frac{q'}{q}}& H_n^{(q)}(A)\ar[r]\ar[r]\ar@{=}[d]  &\cdots  \\
  \cdots \ar[r]&H_{n+1}^{(q)}(A)\ar[r]& H_n^{(p')}(A) \ar[r]^{q} & H_n^{(pq)}(A)
  \ar[r]& H_n^{(q)}(A)\ar[r] &\cdots ,}
$$
where rows are long twosided exact sequences given in Proposition
1.4 (3). The direct limit of these long exact sequences with respect to $p$ yields the required
long twosided exact sequence.
\end{proof}

The homology theory $H^{\mathbb{T}}$ is said to be the $\mathrm{torsion}\;\;
\mathrm{homology}$ of the homology $H$.

\begin{cor}
\label{tfbc} Let $\tau :H\rightarrow \tilde{H}$ be a natural
transformation of homology theories. Then $\tau $ induces natural
transformations $\tau ^{(q)}:H^{(q)}\rightarrow \tilde{H}^{(q)}$,
\;\; $\tau ^{\mathbb{T}}:H^{\mathbb{T}}\rightarrow
\tilde{H}^{\mathbb{T}}$ and $\tau _ \mathbb{Q}:H\otimes
\mathbb{Q}\rightarrow \tilde{H}\otimes \mathbb{Q}$. Furthermore
the following conditions are equivalent.
\begin{enumerate}
    \item $\tau (A)$ is an isomorphism for a $C^*$-Algebra $A$.
    \item $\tau ^{\mathbb{T}}(A)$ and $\tau _
    \mathbb{Q}(A)$ is an isomorphism for a $C^*$-Algebra $A$.
    \item $\tau ^{(q)}(A)$ for all prims $q$ and $\tau _
    \mathbb{Q}(A)$ is an isomorphism for a $C^*$-Algebra $A$.
\end{enumerate}
\end{cor}

\begin{proof} $(1)\cong (2)$ is consequence of the Five Lemma and
following commutative diagram of twosided long exact sequence:
$$
\xymatrix{..\ar[r]&H^{\mathbb{T}}_{n+1}(A)\ar[r]\ar[d]^{\tau
^{\mathbb{T}}(A)} &H_n(A)\ar[r]\ar[d]^{\tau (A)} &H_n(A)\otimes
\mathbb{Q}\ar[r]\ar[d]^{\tau _
    \mathbb{Q}(A)} &H^{\mathbb{T}}_{n}(A)\ar[r]\ar[d]^{\tau ^{\mathbb{T}}(A)} &..\\
..\ar[r]&\tilde{H}^{\mathbb{T}}_{n+1}(A)\ar[r]
&\tilde{H}_n(A)\ar[r] &\tilde{H}_n(A)\otimes \mathbb{Q}\ar[r]
&\tilde{H}^{\mathbb{T}}_{n}(A)\ar[r]
&..}
$$

 $(2)\cong (3)$ is a consequence of the Five Lemma and
following commutative diagrams of the twosided long exact
sequence:
$$
\xymatrix{..\ar[r]&H_{n+1}^{(q)}(A)\ar[r]\ar[d]^{\tau ^{(q)}}
&H^{\mathbb{T}}_{n}(A)\ar[r]\ar[d]^{\tau ^{\mathbb{T}}}
&H^{\mathbb{T}}_{n}(A)\ar[r]\ar[d]^{\tau ^{\mathbb{T}}}
 &H_{n}^{(q)}(A)\ar[r]\ar[d]^{\tau ^{(q)}} &..\\
..\ar[r]                &\tilde{H}_{n+1}^{(q)}(A)\ar[r]
&\tilde{H}^{\mathbb{T}}_{n}(A)\ar[r]
&\tilde{H}^{\mathbb{T}}_{n}(A)\ar[r] &\tilde{H}
^{(q)}_{n}(A)\ar[r] &..}
$$

and
$$
\xymatrix{..\ar[r]&H_{n+1}^{(q)}(A)\ar[r]\ar[d]^{\tau ^{(q)}}
&H_{n}^{(p)}(A)\ar[r]\ar[d]^{\tau ^{(p)}}
&H_{n}^{(pq)}(A)\ar[r]\ar[d]^{\tau ^{(pq)}}
 &H_{n}^{(q)}(A)\ar[r]\ar[d]^{\tau ^{(q)}} &\dots\\
..\ar[r] &\tilde{H}_{n+1}^{(q)}(A)\ar[r]
&\tilde{H}_{n}^{(p)}(A)\ar[r] &\tilde{H}_{n}^{(pq)}(A)\ar[r]
&\tilde{H}_{n}^{(q)}(A)\ar[r] &..}
$$

\end{proof}

\section{Applications to $KK$-theory}
\subsection{Torsion and finite $KK$-theories}

By considering $KK^G(A,-)$ as a homology theory and according to
section 1 we define finite, torsion and rational $KK^G$-theories
for all integer $n$ as follows.

\bigskip

\begin{defn}
\begin{equation}\label{rmodp}
KK^G_n(A,B;\zq)=KK^G_{n-2}(A,B\otimes C_q).
\end{equation}
\end{defn}

\bigskip

\begin{defn}
$$KK^G_n(A,B,T)=\underrightarrow{\lim}_q\;KK^G_n(A,B;\zq)$$.
\end{defn}

\bigskip

\begin{defn}
$$KK^G_{n}(A,B;Q)= KK^G_{n}(A,B)\otimes Q.$$
\end{defn}

\bigskip

Our definitions of finite and rational $KK^G$ differ from the
exisiting definitions of finite and rational $KK$-theories ([3],
23.15.6-7). In effect, here we compare the two versions of the
definitions of finite and rational KK- theories.

1. Let N be the smallest class of separable C¤-algebras with the following properties:

 (N1) N contains field complex numbers;

 (N2) N is closed under countable inductive limits;

 (N3) if $0 \rightarrow A \rightarrow D \rightarrow B \rightarrow
 0$

  is an exact sequence, and two of them are in
N, then so is the third;

 (N4) N is closed under $KK$-equivalence.

Let $D$ be a $C^*$-algebra in N with $K_{0}(D) = Z_{p}, K_{1}(D) =
0$.

 Define

$$
KK_{n}(A;B;Zp) = KK_{n}(A;B\otimes D).$$

As noted in (\cite{bla}), so defined $KKS$-groups are independent
of the choice of $D$. Bellow we show that the above definition is
equivalent to our definition. One has $K_{0}(C_m) = Z_m$ and
$K_{1}(C_{m}) = 0$. This is an easy consequence of the Bott
periodicity theorem and the two-sided long exact sequence

$$
 ....\rightarrow  K_{2}(C_{0}(S^{1})) \rightarrow K_{2}(C_{0}(S^{1})) \rightarrow K_{2}(C_{m}) \rightarrow K_{1}(C_{0}(S^{1})) \rightarrow ...,
$$

since $K_{1}(C_{0}(S^{1})) = 0.$

Therefore our definition of finite $KK$-theory agrees to its
definition in the sense of [1] taking into account the following
isomorphism induced by the Bott periodicity theorem:

$$ KK_{n}(A;B\otimes C_{m}(S^{1})) \simeq KK_{n-2}(A;B\otimes
C_{m}(S_{1})). $$

2. The rational $KK$-theory is defined in ([1], 23.15.6) by the
following manner. Let $D$ be a $C^*$-algebra in $N$ with $K_{0}(D)
= Q, K_{1}(D) = 0$. Define

$$
KK_{n}(A;B;Q) = KK_{n}(A;B \otimes D).
$$

In general $KK_{n}(A;B;Q) \neq KK_{n}(A;B) \otimes Q$ (\cite{bla},
23.15.6). For example,
$$KK(D;C;Q) = Q\;\;\;\text{and}\;\;\;KK(D;C)\otimes Q = 0.$$
This means that our rational $KK$-theory differs from that of
\cite{bla}.

According to results of the previous section one has the following
properties of $q$-finite and torsion $KK$-theories.

\begin{enumerate}
    \item The groups $KK^G(A,B;\zq)$ have Bott periodicity
property and satisfy the excision property relative to both
arguments.
    \item there is a natural twosided exact sequence:
\begin{multline}\label{rexc}
 \cdots  \rightarrow KK^G_n(A,B) \xrightarrow{q\times } KK^G_n(A,B)
 \rightarrow KK^G_n(A,B;\zq) \rightarrow \\\rightarrow KK^G_{n-1}(A,B)
  \xrightarrow{q\times } KK^G_{n-1}(A,B) \rightarrow \cdots
\end{multline}

    \item there is a natural twosided exact sequence:
\begin{multline}\label{rexc1}
 \xrightarrow{\acute{q}}
 KK^G_n(A,B,\mathbb{Z}_{pq})
 \xrightarrow{\grave{p}} KK^G_n(A,B;\zq) \rightarrow \\\rightarrow
 KK^G_{n-1}(A,B,\mathbb{Z}_p )
  \xrightarrow{\acute{q}} KK^G_{n-1}(A,B,\mathbb{Z}_{pq}) \rightarrow
\end{multline}
    \item There is a associative product
\begin{equation}
\label{podass}KK^G_n(A,B;\mathbb{Z}_{p})\otimes
KK^G_m(A,B;\mathbb{Z}_{q})\rightarrow
KK^G_{n+m-2}(A,B;\mathbb{Z}_{pq})
\end{equation}
 \item there is a natural twosided exact sequence:
\begin{multline}\label{rexc2}
\dots  \rightarrow
 KK^G_n(A,B,\t)
\xrightarrow{\grave{q}} KK^G_n(A,B;\t) \xrightarrow{\breve{p}}
 KK^G_{n}(A,B,\mathbb{Z}_q )\\
  \rightarrow KK^G_{n-1}(A,B,\t) \rightarrow \dots
\end{multline}
\item there is a natural twosided exact sequence:
\begin{multline}\label{rexc2}
\dots  \rightarrow
 KK^G_n(A,B)
\xrightarrow{r } KK^G_n(A,B;\q) \xrightarrow{\breve{t}}
 KK^G_{n}(A,B,\t)\\
  \rightarrow KK^G_{n-1}(A,B,) \rightarrow \dots
\end{multline}
\end{enumerate}

In addition there is an associative product
\begin{equation*}
\label{vaxel} KK^G_n(A,B;\q)\otimes KK^G_m(B,C;\q)\rightarrow
KK^G_{n+m}(B,C;\q)
\end{equation*}
Tensor product is considered over ring of integers. The product is
a composition of the isomorphism:
\begin{multline}\label{asqr}
    (KK^G_{n}(A,B)\otimes \mathbb{Q})\otimes (KK^G_{n}(B,C)\otimes \mathbb{Q})\cong\\
   \cong
   (KK^G_{n}(X;A,B)\otimes KK^G_{n}(X;B,C))\otimes
   (\mathbb{Q}\otimes \mathbb{Q}),
\end{multline}
which is the composition of the twisting and associativity
isomorphisms of tensor product, and a homomorphism
\begin{equation}
(KK^G_{n}(A,B)\otimes KK^G_{n}(B,C))\otimes
   (\mathbb{Q}\otimes \mathbb{Q})\longrightarrow KK^G_{n}(A,C)\otimes
   \mathbb{Q}
\end{equation}
defined by a map $(f\otimes r)\otimes (f'\otimes r')\mapsto
(f\cdot f)'\otimes rr'$, where $f\cdot f$ is Kasparov product of
$f$ and $f'$.

\bigskip

Thus we can form an additive category $KK^G_\mathbb{Q}$, where
$GC^*$-algebras are objects and the group of morphisms from $A$ to
$B$ is given by the equality
\begin{equation}\label{rR}
KK^G_n(A,B;\mathbb{Q})=KK^G_{n}(X;A,B)\otimes \mathbb{Q}.
\end{equation}

\bigskip

There is a natural additive functor
$$
Rat:KK^G\longrightarrow KK^G_\mathbb{Q}
$$
which is identity on objects, and on morphisms is defined by the map
$f\mapsto f\otimes 1$. It is clear that $Rat$ is an additive
functor.

The result below says that $KK^G_\mathbb{Q}$ is a
 bivariant theory on the category of separable
$GC^*$-algebra and it is said to be the \textit{rational}
$KK^G$-theory.

\begin{thm}
The additive category $KK^G_\mathbb{Q}$ is a bivariant theory on
the category of separable $GC^*$-algebras, i.e. has all
fundamental properties of usual bivariant $KK$-theory. Besides,
$KK^G_\mathbb{T}$ is a bimodule on the category $KK^G$ such that
it is cohomological functor relative the first argument and
homological functor relative to the second argument satisfying the
Bott periodicity property.
\end{thm}

\begin{proof}
This is easy consequence of the fact that $\q$ is a flat
$\mathbb{Z}$-module and the tensor product on a flat module
preserves exactness.
\end{proof}

\subsection{A look at Baum-Connes Conjecture}

In the formulation of Baum-Connes Conjecture a crucial role play the
groups $K^{top}_n(G,A)$,  so called the topological K-theory of
$G$ with coefficients in $A$, and the homomorphism
$$
\mu _A : K^{top}_n(G,A)\rightarrow K_n(G \ltimes _r A),
$$
which is called the Baum-Connes assembly map. The Baum-Connes
Conjecture for $G$ with coefficients in $A$ asserts that this map is
an isomorphism. Note that $K^{top}_n(G,\;-)$ and $K_n(G\ltimes
_r\;-)$ are homology theories  in the sense that we have defined
in the first section (cf. \cite{mene}).  Therefore we have
rational, torsion and finite versions of Baum-Connes Conjecture:
\begin{itemize}
    \item (Rational version) the assembly map
    $$\mu _A\otimes id_{\q} : K^{top}_n(G,A)\otimes  \q \rightarrow K_n(G \ltimes _r A)\otimes  \q
$$ is an isomorphism;
    \item (Finite version) the $q$-finite assembly map
    $$\mu _A^{(q)} : K^{top}_n(G,A;\zq) \rightarrow K_n(G \ltimes _r A;\zq)\otimes
$$ is an isomorphism;
    \item (Torsion version) the torsion assembly map
    $$\mu _A^{(q)} : K^{top}_n(G,A;\mathbb{T)}) \rightarrow K_n(G \ltimes _r A;\mathbb{T)})
$$ is an isomorphism.
\end{itemize}

According to Corollary \ref{tfbc}, we have the following theorem
\begin{thm}
\label{bcceq} The following Conjectures are equivalent.
\begin{enumerate}
    \item Baum-Connes Conjecture;
    \item Baum-Connes rational and torsion Conjectures;
    \item Baum-Connes rational and $q$-finite Conjectures for all
    primes.
\end{enumerate}
\end{thm}

\section{Remarks on finite algebraic and topological $K$-theories}

We begin with some preliminary definitions and properties.
In \cite{brw}, Browder has defined algebraic
$K$-theory of an unital ring with coefficients in $\mathbb{Z}/q$,
$q\geq 2$ as follows:
$$
K_n(R;\mathbb{Z}/q)=\pi _n(BGL(R)^+;\mathbb{Z}/q)
$$
by using so called homotopy groups with coefficients
in $\mathbb{Z}/q$.

\textit{Remark}. Bellow "Algebraic $K$-theory of an
unital ring with coefficients in $\mathbb{Z}/q$"  will be replaced
by "$q$-finite algebraic $K$-theory of an unital ring".

For our purposes we use equivalent definition used in \cite{arin}:
$$
K_{n+1}^a(R;\mathbb{Z}/q)=\pi _n(F_q(BGL(R)^+)).
$$
Here, in general, $F_q(X)$ is defined as the homotopy fiber of the
$q$-power map of a loop space $X=\Omega Y$ (see \cite{arin}).

There exists similar interpretation for $q$-finite topological
$K$-theory of $C^*$-al\-
gebras.
 If $A$ is an unital $C^*$-algebra. Then $GL(A)$ has the
 standard topology induced by the norm in $A$. Denote this topological
 group by $GL^t(A)$. It is known that $GL^t(A)$ and $\Omega
 B(GL^t(A))$ are homotopy equivalent spaces. Therefore topological
 $K$-groups may be defined equivalently by the equality
 $$
K_n^t(A)=\pi _n(B(GL^t(A))),\;\;\;\;n\geq 1.
 $$
Therefore one can define the $q$-finite topological $K$-theory as
follows:
$$
K_{n+1}^t(R;\mathbb{Z}/q)=\pi _n(F_q(B(GL^t(R)))).
$$
We have natural, up to homotopy, maps
$$B(GL(A))^+\rightarrow B(GL^t(A))$$ and $$F_qB(GL(A))^+\rightarrow
F_qB(GL^t(A)).$$
  Therefore we have natural homomorphisms
$$
\alpha _n:K^a_n(A)\rightarrow
K^t_n(A)\;\;\;\;\;\;\;\;\text{and}\;\;\;\;\;\;\;\alpha
_{n,q}:K^a_n(A,\mathbb{Z}/q)\rightarrow
K^t_n(A,\mathbb{Z}/q),\;\;\;\;\;\;
$$
$n\geq 1,\;\;q\geq 2.$
\begin{prop}
Let $A$ be a $C^*$-algebra and $\mathcal{K}$ be a $C^*$-algebra of
compact operators on a separable Hilbert space. Then the natural
homomorphisms
$$
\varepsilon ^{-1}\alpha _{n,q}:K^a_n(A\otimes
\mathcal{K},\mathbb{Z}/q)\rightarrow
K^t_n(A,\mathbb{Z}/q)\;\;\;n\geq 1,\;\;q\geq 2,
$$
are isomorphisms, where $\varepsilon
:K^t_n(A;\mathbb{Z}_q)\xrightarrow{\cong} K^t_n(A\otimes
\mathcal{K};\mathbb{Z}_q)$ is the isomorphism of stability for the
finite topological $K$-theory of $C^*$-algebras.
\end{prop}

\begin{proof}
It is enough to show that the homomorphism $$\alpha _{n,q}
:K^a_n(A\otimes \mathcal{K};\mathbb{Z}_q)\rightarrow
K^t_n(A\otimes \mathcal{K};\mathbb{Z}_q)$$ is an isomorphism. To
this end consider the following commutative diagram
$$
\xymatrix{\cdots \ar[r]& K^a_{n+1}(A\otimes
\mathcal{K};\mathbb{Z}_q)\ar[r] \ar[d]^{\alpha _{n+1,
q}}&K^a_n(A\otimes \mathcal{K})\ar[r]^{\times q}\ar[d]^{\alpha _n}
&K^a_n(A\otimes \mathcal{K})\ar[r]\ar[d]^{\alpha _n} & \cdots \\
\cdots \ar[r]& K^t_{n+1}(A\otimes \mathcal{K};\mathbb{Z}_q)\ar[r]
&K^t_n(A\otimes \mathcal{K})\ar[r]^{\times q} &K^t_n(A\otimes
\mathcal{K})\ar[r] & \cdots }.
$$

Since the natural homomorphisms $\alpha _n:K^a_n(A\otimes
\mathcal{K})\rightarrow K^t_n(A\otimes \mathcal{K})$ are
isomorphisms for any integer $n$ \cite{suw}, then by the Five
Lemma the homomorphism $$K^a_{n}(A\otimes
\mathcal{K};\mathbb{Z}_q)\rightarrow K^t_{n}(A\otimes
\mathcal{K};\mathbb{Z}_q) $$ is an isomorphism too for all $n\geq
2$.

\end{proof}

\section{Browder-Karoubi-lambre 's theorem for finite $KK$-theory}

One has the following interpretation of the $q$-finite topological
$K$-theory.
\begin{prop}
There are a natural isomorphisms
$$
K^t_{n}(A;\mathbb{Z}/q)\;\cong \;K_{n-2}^t(A\otimes C_q),
$$
for all $n\geq 1$ and $q\geq 2$.
\end{prop}

\begin{proof} Since classifying space construction has functorial property,
 according to the functorial
property of the functor $B(GL^t(-))$ and the commutative diagram
$$
\xymatrix{A\otimes C_q\ar[r]\ar[d] &A\otimes
C_0(S^1)\otimes C[0;1)\ar[d]\\
A\otimes C_0(S^1)\ar[r] & A\otimes C_0(S^1),}
$$

one gets the commutative diagrams
$$
\xymatrix{B(GL^t(A\otimes C_q))\ar[r]\ar[d] &B(GL(A\otimes
C_0(S^1)\otimes C[0;1)))\ar[d]\\
B(GL^t(A\otimes C_0(S^1)))\ar[r] & B(GL(A\otimes C_0(S^1)))}
$$
and
$$
\xymatrix{F_q(\Omega B(GL^t(A)))\ar[r]\ar[d] &\Omega B(GL^t(A))^{[0,1)}\ar[d]\\
\Omega B(GL^t(A)\ar[r]^{q} &\Omega B(GL^t(A))}
$$
Since the second diagram is universal, there exists a natural map
$$
\chi :B(GL^t(A\otimes C_q))\rightarrow F_q(\Omega B(GL^t(A))).
$$
Therefore one has a natural homomorphism
$$
\pi _n \chi :\pi _n(B(GL^t(A\otimes C_q)))\rightarrow \pi
_{n}(\Omega F_q( B(GL^t(A))))
$$
Thus there is a natural homomorphism
$$
\chi _n:K_{n}^t(A\otimes C_q)\rightarrow
K_{n+2}^t(A,\mathbb{Z}/q).
$$
Now, consider the following commutative diagram
$$
\xymatrix{\cdots \ar[r]& K_{n}^t(A\otimes C_q)\ar[r]
\ar[d]^{\chi_n}&K_{n}^t(A\otimes C_0(S^1))\ar[r]^{\times
q}\ar[d]^{=}
&K_{n}^t(A\otimes C_0(S^1))\ar[r]\ar[d]^{=} & \cdots \\
\cdots \ar[r]& K_{n+2}^t(A,\mathbb{Z}/q)\ar[r]
&K^t_{n+1}(A)\ar[r]^{\times q} &K^t_{n+1}(A)\ar[r] & \cdots}
$$
According to the Five Lemma, one concludes that $\chi _n$ are
isomorphisms, $n\geq 1$.
\end{proof}

Let $H:\mathcal{C}^*\rightarrow Ab$ be a functor, where
$\mathcal{C^*}$ is the category of unital $C^*$- algebras and their
homomorphisms (non-unital). Then
\begin{enumerate}
    \item if the inclusion in the upper left corner  $A\hookrightarrow M_n(A)$ induces
isomorphism $H(A)\cong H(M_n(A))$, $H$ is said to
be matrix invariant functor.
    \item if $H$ commutes with direct system of $C^*$-algebras, $H$ is said to
be continuous.
\end{enumerate}
For a given matrix invariant and continuous functor $H$ there
exists an extension $\mathcal{H}$ of it on the category of small
additive $C^*$-categories $Add\;C^*$ such that the following
diagram
$$
\xymatrix{C^*\ar[rr]^{proj_f }\ar[dr]_H & &Add\;C^*\ar[dl]^{\mathcal{H}}\\
&Ab &}
$$
commutes, where $proj_f$ is a functor which sends unital
$C^*$-algebra $A$ to the additive $C^*$-category of finitely
generated projective $A$-modules. The functor $\mathcal{H}$ is
defined by the following manner (cf.\cite{hig1}, \cite{kan1}).

First note that the functor $H$ is a inner invariant functor (see
Lemma 2.6.12 in \cite{hig2}).
 Let $\mathbb{A}$ be an additive $C^*$-category. Set
$\mathcal{L}(a)=\hom _{\mathbb{A}}(a,a)$, $a\in ob \mathbb{A}$.
Let us write $a\leq a'$ if there is an isometry $v:a\rightarrow
a'$ in $A$, i.e. $v^*v=id_a$. The relation ''$a\leq a$'' makes the
set of objects into a directed set.

Any isometry $v:a\rightarrow a'$ in $A$ defines a $*$-homomorphism
of $C^{*}$-algebras
$$
{\rm Ad}(v):\mathcal{L}(a)\rightarrow \mathcal{L}(a')
$$
by the rule $x\mapsto vxv^{*}$.

Using technics from \cite{hig1}, one has the following. Let
$v_1:a\rightarrow a'$ and $v_2:a\rightarrow a'$ be two isometries
in $\mathbf{A}$. Then the homomorphisms
$$
{\rm Ad}_{*}v_1,\;\;{\rm Ad}_{*}v_2:H(\mathcal{L}(a))\rightarrow
H(\mathcal{L}(a'))
$$
are equal. Indeed, let $u=\left(
\begin{array}{cc}
  0 & 1 \\
  1 & 0 \\
\end{array}
\right)$ be the unitary element in an unital $C^{*}$-algebra
$M_2(\mathcal{L}(a'))$. Since $H$ is a matrix invariant functor, it
is {\em inner invariant} functor too (see Lemma 2.6.12 in
\cite{hig2}), i.e. the homomorphism $H(ad(u))$ is the identity
map. Therefore, the maps
$$
x\mapsto \left(
\begin{array}{cc}
x & 0 \\
0 & 0
\end{array}
\right) \;\;\;{\rm and}\;\;\;x\mapsto \left(
\begin{array}{cc}
0 & 0 \\
0 & x
\end{array}
\right)
$$
sending $\mathcal{L}(a')$ into $M_2(\mathcal{L}_A(I)(a'))$, induces
the same isomorphisms after applying the functor $H$. It is clear
that the homomorphism $\nu _{*}^{aa'}=H(\nu ^{aa'})$ is not
depending on the choice of an isometry $\nu ^{aa'}:a\rightarrow
a'$. Therefore one has a direct system
$\{H(\mathcal{L}(a)),\nu_{*}^{aa'})\}_{a,a'\in obA}$ of abelian
groups.
\begin{defn}
\label{waa}Let $\mathbb{A}$ be an additive small $C^*$-category.
Then by definition
$$
\mathcal{H}(\mathbb{A})=\underrightarrow{\lim}\;H(\mathcal{L}(a)).
$$
\end{defn}
So defined functor makes commutative the above diagram. That follows
from the matrix invariant and continuous properties of
$H$ and is a simple exercise  (see \cite{kan1}).

Since the functors $K^t(-;\mathbb{Z}/q)$ have the above mentioned
properties, one can define the $q$-finite topological $K$-theory for
an additive $C^*$-category $\mathbb{A}$ by setting
$$
K^t_n(\mathbb{A};\mathbb{Z}/q)=\underrightarrow{\lim}\;K^t_n(\mathcal{L}(a);\mathbb{Z}/q).
$$
This definition is in accordance with other definitions of
$q$-finite topological $K$-theories because of the matrix
invariant and continuous properties. Therefore we get a
generalization of Browder-karoubi-Lambre 's theorem for small additive
$C^*$-categories.

\begin{prop}
\label{badt} Let $\mathbb{A}$ be a small additive $C^*$-category.
Then, for all $n\in \mathbb{Z}$,
\begin{enumerate}
    \item $\;\;\;\;$ $q\cdot K_n^t (\mathcal{A}, Z/q) = 0,$ $\;\;\;\;\;\;$ if $q-2$ is not divided by $4$;
    \item $\;\;\;\;$ $2q \cdot K_n ^t(\mathcal{A}, Z/q) = 0,$ $\;\;\;\;\;$ if $4$ divides $q-2$.
    \end{enumerate}
\end{prop}

\begin{proof}
It is consequence of Proposition 4.1.
\end{proof}

The next step is to give an interpretation of $q$-finite
$KK^G$-theory as topological $K$-theory of the
additive $C^*$-category $Rep_G(A,B)$. Such an interpretation
exists for $KK^G$-theory, where $G$ is a compact metrizable group
\cite{kan1}.

\begin{thm}
\label{gbtt} Let $A$ and $B$ be, respectively, separable and
$\sigma $-unital $G-C^*$-algebras, real or complex; and $G$ be
metrizable compact group. Then, for all integer $n$ and $q\geq 2$,
there exists a natural isomorphisms
$$
KK_n^G (A,B;\zq)\cong K^t_{n+1}(\mathrm{Rep}(A,B);\mathbb{Z}/q),
$$
\end{thm}

When $G$ is locally compact group, the proof is more complicated and
 this case will be investigated in the further paper.

First we recall the definition of the $C^*$-category $Rep(A,B)$. This
category was constructed in \cite{kan1}.

Let $\mathcal{H}_G(B)$ be the additive $C^*$-category of countably
generated right Hilbert $B$-modules equipped with a $B$-linear,
norm-continuous $G$-action over a fixed compact second countable
group $G$ \cite{kas1}. Note that the compact group acts on the
morphisms by the following rule: for $f:E\rightarrow E'$ the
morphism $gf:E\rightarrow E'$ is defined by the formula
$(gf)(x)=g(f(g^{-1}(x)))$.

The category $\mathcal{H}_G(B)$ contains the class of compact
$B$-homomorphisms \cite{kas1}. Denote it by $\mathcal{K}_{G}(B)$.
Known properties of compact $B$-homomorphisms imply that
$\mathcal{K}_{G}(B)$ is a $C^*$-ideal \cite{clr} in
$\mathcal{H}_G(B)$.

Objects of the category $Rep(A,B)$ are pairs of the form
$(E,\varphi)$, where $E$ is an object in $\mathcal{H}_G(B)$ and
$\varphi:A\rightarrow\mathcal{L}(E)$ is an equivariant
$*$-homomorphism. A morphism $f:(E,\phi)\rightarrow(E',\phi')$ is
a $G$-invariant morphism $f:E\rightarrow E'$ in $\mathcal{H}_G(B)$
such that
$$
f\phi(a)-\phi'(a)f\in\mathcal{K}_G(E,E')
$$
for all $a\in A$. The structure of a $C^*$-category is inherited
from $\mathcal{H}_G(B)$. It is easy to see that $Rep(A,B)$ is an
additive $C^{*}$-category, not idempotent-complete.

Now, we are ready to construct our main $C^*$-category, that is
$\mathrm{Rep}(A,B)$. Its objects are triples $(E,\phi,p)$, where
$(E,\phi)$ is an object and $p:(E,\phi)\rightarrow (E,\phi)$ is a
morphism in $Rep(A,B)$ such that $p^*=p$ and $p^2=p$. A morphism
$f:(E,\phi,p)\rightarrow(E',\phi',p')$ is a morphism
$f:(E,\phi)\rightarrow(E',\phi')$ in $Rep(A,B)$ such that
$fp=p'f=f$. In detail, $f$ must satisfy
\begin{equation}
f\phi(a)-\phi'(a)f\in \mathcal{K}(E,F) \textrm{ and } fp=p'f=f.
\end{equation}
So, by definition
$$
\mathrm{Rep}(A,B)=\widetilde{Rep(A,B)}.
$$
The structure of a $C^*$-category on $\mathrm{Rep}(A,B)$ comes
from the corresponding structure on $Rep(A,B)$.

\begin{proof}(\textit{of the theorem} \ref{gbtt}) The following isomorphisms
\begin{equation*}
\theta _n^a: K^a_n(\mathrm{Rep}(A;B))\simeq KK^G_{n-1}(A;B),
\end{equation*}
and
\begin{equation*}
\theta _n^t:K^t_n(\mathrm{Rep}(A;B))\simeq KK^G_{n-1}(A;B),
\end{equation*}
was proved in \cite{kan1}. According to the definition of the
finite $KK^G$-groups and these isomorphisms, in particular, we
have the following result for finite $KK^G$-theory:

Let $A$ and $B$ be, respectively, separable and  $\sigma $-unital
$G-C^*$-algebras. Then
\begin{equation}\label{rept}
KK^G_{n}(A,B; \mathbb{Z}_q)\cong K^t_{n-1}(\mathrm{Rep}(A;B\otimes
C_q))\cong K^a_{n-1}(\mathrm{Rep}(A;B\otimes C_q)).
\end{equation}

Therefore it is enough to show that
$$
K^t_{n+1}(\mathrm{Rep}(A,B);\mathbb{Z}/q)\cong
K^t_{n-1}(\mathrm{Rep}(A;B\otimes C_q)).
$$
Note that
$$
K^t_{n-1}(\mathrm{Rep}(A,B\otimes C_q)\cong
\underrightarrow{\lim}_{a\in \mathrm{Rep}(A,B\otimes
C_q)}\;\;K^t_{n-1}(\mathcal{L}(a))
$$
and
$$
K^t_{n+1}(\mathrm{Rep}(A,B;Z_q))=\underrightarrow{\lim}_{b\in ob
\mathrm{Rep}(A,B)}\;\;K^t_{n-1}(\mathcal{L}(b)\otimes C_q).
$$
So it is enough to compare the right-hand sides.

Consider $\mathrm{Rep}(A,B)\otimes C_q$ as the $C^*$-tensor
product of $C^*$-categoroids in the sense of \cite{kan1} (or as
non-unital $C^*$-categories in the sense of \cite{mitch}).

There is a natural (non-unital) functor
$$
\nu :Rep(A,B)\otimes C_q\rightarrow \mathrm{Rep}(A,B\otimes C_q)
$$
defined by maps:
\begin{enumerate}
    \item $b=(\varphi , E, p)\mapsto \varphi \otimes id_{C_q}, E\otimes
C_q, p\otimes \id_{C_q})=a_b$ on objects;
    \item $f\mapsto f\otimes id_{C_q}$ on morphisms.
\end{enumerate}
 One has induced morphism of direct systems of abelian groups
$$
\{\nu _a\} :\{K^t_n(\mathcal{L}(a)\otimes C_q)\}\rightarrow
\{K^t_n(\mathcal{L}(b)\},
$$
where $\nu _a:K^t_n(\mathcal{L}(a)\otimes C_q)\rightarrow
K^t_n(\mathcal{L}(a_b)$ is induced by $\nu$. Therefore one has a
natural homomorphism
$$
\bar{\nu} _n:K^t_{n+1}(\mathrm{Rep}(A,B);\mathbb{Z}/q)\rightarrow
K^t_{n-1}(\mathrm{Rep}(A;B\otimes C_q)).
$$
Then comparing the  two twosided exact sequences
$$
\xymatrix{\cdots \ar[r]& K^t_{n+1}(\mathrm{Rep}(A;B) ,
\mathbb{Z}_q)\ar[r]\ar[d]^{\bar{\nu} _n}
&K^t_n(\mathrm{Rep}(A;B))\ar[r]^{\times q}\ar[d]^=
&K^t_n(\mathrm{Rep}(A;B))\ar[r]\ar[d]^= & \cdots \\\cdots \ar[r]&
K^t_{n-1}(\mathrm{Rep}(A;B\otimes C_q))\ar[r]
&K^t_n(\mathrm{Rep}(A;B))\ar[r]^{\times q}
&K^t_n(\mathrm{Rep}(A;B))\ar[r] & \cdots}
$$
one concludes that $ \bar{\nu} $ is an isomorphism.
\end{proof}

Now,we show the Browder-Karoubi-Lambre's theorem for finite $KK^G$-theory .
\begin{thm}
\label{gbtt} Let $A$ and $B$ be, respectively, separable and
$\sigma $-unital $G-C^*$-algebras, real or complex; and $G$ be
metrizable compact group. Then, for all $n\in \mathbb{Z}$,
\begin{enumerate}
    \item $\;\;\;\;$ $q\cdot KK_n^G (A,B;\zq) = 0,$ $\;\;\;\;\;\;$ if $q-2$ is not divided by $4$;
    \item $\;\;\;\;$ $2q \cdot KK_n^G (A,B;\zq) = 0,$ $\;\;\;\;\;$ if $4$ divides $q-2$.
    \end{enumerate}
\end{thm}

\begin{proof}
Follows from Propositions 4.1, 5.1 and 5.3, from Theorem 5.4 and from the Browder-Karoubi-Lambre's theorem for algebraic K-theory.
\end{proof}



\begin{thebibliography}{99}
\bibitem{arin} Arlettaz D. Inassaridze Finite $K$-theory spaces
Math. Proc. Camb. Phil.Soc. (2005), 139. 261-286.

\bibitem{brw} Browder W. Algebraic $K$-theory with coefficients
$\mathbb{Z}/p$, in Geometric Applications of Homotopy Theory I,
Lect. Notes in Math. 657 (Springer, 1978), 40-84.

\bibitem{bla}  Blackadar B. $K$-\textit{theory for Operator Algebras}, M.S.R.I.
Publ. 5, Springer-Verlag, (1986).

\bibitem{clr}
P. Chez, R. Lima and J. Roberts, \textit{$W^{*}$-categories},
Pacific J. Math. $\mathbf{120}$ (1985), No 1, 79-109.

\bibitem{cumero} Cuntz J., Meyer R., Rosenberg J., \textit{Topological and Bivariant K-Theory}, Birkhäuser, Basel (2007).

\bibitem{hig2}  Higson N. \textit{Algebraic $K$-theory of stable $C^*$-algebras},
Adv. Math., v.67, (1988) 1-140.

\bibitem{hig1}  Higson N.\ $C^{*}$-\textit{algebra extension theory and duality}
, J.Funct. Anal., v.129, (1995), 349-363.

\bibitem{kan1}  Kandelaki T., \textit{Algebraic K-theory of Fredholm modules and
KK-theory} Vol. 1, (2006),  No. 1, 195-218.

\bibitem{kas1}  Kasparov G. \textit{Hilbert $C^{*}$-modules: Theorems of
stinespring and voiculescu}, J. Operator theory, v. 4, (1980),
133-150.

\bibitem{mene}Meyer R., Nest R. The Baum-Connes Conjecture via Localisation of Categories
arXiv:math/0312292 .

\bibitem{mitch} Mitchener P. C*-\textit{categories} Proceedings of the London
Mathematical Society, volume 84 (2002), 375-404.

\bibitem{suw}  Suslin A., Wodzicki M. \textit{Excision in algebraic
$K$-theory}, Ann. Math. v.136, No.1, (1992), 51-122.

\bibitem{weid} Weidner J. \textit{Topological invariant for Generalized
Operatorial Algebras} Dissertation, Heidelberg (1987).
\end{thebibliography}
\end{document}